\documentclass[12pt,a4paper]{amsart}
\usepackage{graphicx,amssymb}
\input xy
\xyoption{all}
\usepackage[all]{xy}
\usepackage{hyperref}

\newtheorem{thm}{Theorem}[section]
\newtheorem{cor}[thm]{Corollary}
\newtheorem{lem}[thm]{Lemma}
\newtheorem{prop}[thm]{Proposition}
\theoremstyle{definition}

\newtheorem{rem}[thm]{Remark}

\numberwithin{equation}{section}

\newcommand{\ZZ}{\mathbb Z}
\newcommand{\CC}{\mathbb C}
\newcommand{\PP}{\mathbb P}

\newcommand{\hra}{\hookrightarrow}
\newcommand{\ra}{\rightarrow}

\newcommand{\cA}{\mathcal{A}}

\newcommand{\tC}{\widetilde{C}}

\newcommand{\cM}{\mathcal{M}}

\newcommand{\cO}{\mathcal{O}}

\newcommand{\cR}{\mathcal{R}}
\newcommand{\cS}{\mathcal{S}}

\DeclareMathOperator{\Pic}{Pic}
 \DeclareMathOperator{\Nm}{{Nm}}

 \DeclareMathOperator{\divi}{div}

\begin{document}

\title[The trigonal construction]{The trigonal construction in the ramified case}
\author[H. Lange and A. Ortega]{ Herbert  Lange and  Angela Ortega}
\address{H. Lange \\ Department Mathematik der Universit\"at Erlangen \\ Germany}
\email{lange@mi.uni-erlangen.de}
              
\address{A. Ortega \\ Institut f\" ur Mathematik, Humboldt Universit\"at zu Berlin \\ Germany}
\email{ortega@math.hu-berlin.de}


\begin{abstract} 
To every double cover ramified in two points of a general trigonal curve of genus $g$, one can 
associate an  \'etale double cover of a tetragonal curve of genus $g+1$. We show that the corresponding Prym 
varieties are canonically isomorphic as principally polarized abelian varieties. 
\end{abstract}

\maketitle

\section{Introduction}

Let $\cR_{g}^{tr}$ denote the moduli space of  non-trivial \'etale double coverings of smooth
trigonal curves of genus $g$ and $\cM^{tet}_{g-1,0}$,  the open set of  the moduli space of 
tetragonal curves of genus $g-1$ consisting of tetragonal curves whose fibres of the 4:1 map have at least one \'etale
point. The classical trigonal construction due to Recillas \cite{r} gives a canonical 
isomorphism
$$
\cR_{g}^{tr} \ra \cM^{tet}_{g-1,0}
$$
such that for every covering $\widetilde C \ra C$ of $\cR_{g}^{tr}$ the Prym variety is 
isomorphic to the Jacobian of its image in $\cM^{tet}_{g-1,0}$ as principally polarized abelian 
varieties.\\

The aim of this paper is to show that a similar statement is valid in the case of double covers  over trigonal curves with 
two ramification points. If $f: \widetilde C \ra C$ is a double cover of smooth curves ramified exactly at two points, 
the Prym variety of the cover, which we denote by $P(f)$ or $P(\widetilde C/C)$, is a principally polarized 
abelian variety (ppav). Apart from the \'etale case,  it is the only way to obtain a ppav from a covering between curves\footnote{
With the exception of non-cyclic triple coverings over a genus 2 curve, whose Prym variety is also a ppav. }.
In the sequel, a {\it ramified double cover} will always  mean a double covering ramified at 
exactly two points.
We denote by $\cR b^{tr}_g$ the moduli space of ramified double covers 
$f: \widetilde C \ra C$ of smooth trigonal covers $h: C \ra \PP^1$ with $g$ the genus of $C$
and the additional property that 
the branch locus of $f$ is disjoint from the ramification locus of $h$.

We call an element $\widetilde C \stackrel{f}{\ra} C \stackrel{h}{\ra} \PP^1$ of $\cR b^{tr}_g$
{\it special} if the branch locus of $f$ is contained in a fibre of $h$ and {\it general} otherwise. 
Let $\cR b^{tr}_{g,sp}$ denote the closed subset of $\cR b^{tr}_g$ consisting of special coverings
and $\cR b^{tr}_{g,gen}$ its complement consisting of general coverings. Moreover, let $\cM_{g,1}^{tet}$ denote the subspace of the moduli space of smooth tetragonal curves of genus $g$
as defined in Section 3 and let  $\cR^{tet}_{g}$ the moduli space of \'etale double covers of smooth 
tetragonal curves of genus $g$.
Then our main theorem is (Theorem \ref{thm4.3} and \ref{thm5.1}),

\begin{thm} \label{mainthm}
\begin{enumerate} 
\item[(a)] There is a canonical isomorphism  
$$
 \cR b^{tr}_{g,sp} \ra  \cM_{g,1}^{tet}.
$$
If $\widetilde C \stackrel{f}{\ra} C \stackrel{h}{\ra} \PP^1$ is an element of  $\cR b^{tr}_{g,sp}$
and $X'$ the corresponding smooth tetragonal cover, we get an isomorphism of principally polarized 
abelian varieties  
$$
P(f) \stackrel{\simeq}{\ra} JX'.
$$
\item[(b)]
There is a canonical map
$$
\cR b^{tr}_{g,gen} \ra  \cR_{g+1}^{tet}.
$$
If $\widetilde C \stackrel{f}{\ra} C \stackrel{h}{\ra} \PP^1$ is an element of  $\cR b^{tr}_{g,gen}$
and $\pi: Y \ra X$ the corresponding \'etale double cover, then the principally polarized abelian varieties
$(P(f), \Xi_f)$ and $(P(\pi), \Xi_\pi)$ are canonically isomorphic.
\end{enumerate}
\end{thm}

Furthermore, in the case (b) of the theorem the image of the map is contained
in the subspace  $\cR_{g+1,2}^{tet}$, the locus of \'etale double coverings over tetragonal curves $X$, such that the 4:1 map $k: X \ra \PP^1$ 
has exactly two fibres consisting of two simple ramification points (Proposition \ref{p5.2}). Note that $\cR b^{tr}_{g,gen}$ and 
$\cR_{g+1,2}^{tet}$ are of the same dimension (see Remark 
\ref{r5.3}). We do not know the exact image of the map between these moduli spaces nor whether it is generically injective. \\

In Section 2 we define for every $\widetilde C \ra C \ra \PP^1$ the corresponding covering 
$Y \ra X$ with $X$ tetragonal and work out its geometric properties. In Section 3 we recall
a special case of Donagi's extension of the trigonal construction which is used in Section 4
for the proof of part (a) of the theorem. Finally in Section 5 we give the proof of part (b).
 \bigskip

{\it Aknowledgements}. We would like to thank  Andrey Soldatenkov who helped us with the translation of \cite{da}.

\section{Double covers of trigonal covers}

Let $C$ be a smooth trigonal curve of genus $g \ge 3$ with trigonal cover $h: C \ra \PP^1$. 
According to Hurwitz formula  the ramification divisor $R_h$ of $h$ is of degree $2g+4$. Let
$$
f: \widetilde C \ra C
$$
be a double cover branched over 2 points $p_1,p_2$ of $C$. We assume that $p_1$ and $p_2$ 
are disjoint from $R_h$. Let $C^{(3)}$ and $\widetilde C^{(3)}$ the third symmetric products of 
$C$ and $\widetilde C$ respectively. Let $\PP^1 \simeq g_3^1 \hra C^{(3)}$ be the natural embedding
of the trigonal linear system. We define the variety $Y$ by the following left hand cartesian diagram
\begin{equation} \label{d2.1}
\xymatrix{
 Y  \ar@{^{(}->}[r] \ar[d]_{8:1}^{\widetilde k} & \widetilde C^{(3)}  \ar[d]_{8:1}^{f^{(3)}} \\
g^1_3 \ar@{^{(}->}[r] & C^{(3)} 
}
\end{equation}
 The fibre of $\widetilde k$ over a point $a \in \PP^1$ consists of the 8 sections $s$
of $f$ over $a$:
$$
s: h^{-1}(a) \ra f^{-1}h^{-1}(a) \quad \mbox{with} \quad f \circ s = id.
$$
(In \cite{d} Donagi denotes $Y$ by $f_*\widetilde C$, since considering $\widetilde C$ as a local system on $C$, this is just the push forward local system 
on $\PP^1$). Let denote $\iota$ de involution on $\tC$ associated to $f$. There are 2 structures on $Y, $
an involution denoted also by $\iota$: 
$$ 
\iota: Y \ra Y,  \quad q_1 + q_1 + q_3 \mapsto \iota(q_1)+\iota(q_2)+ \iota(q_3),
$$
and an equivalence relation: two sections 
$$
s_1,s_2: h^{-1}(a) \ra f^{-1}h^{-1}(a)
$$
are called equivalent if they differ by an even number of changes $q \mapsto \iota(q)$. This defines a branched double cover 
$$
O:= O(h \circ f) \ra \PP^1,
$$
called the orientation cover of $h \circ f$. For $n=3$ we have the cartesian diagram (see \cite[Section 2.1
and Lemma 2.1]{d}):
\begin{equation} \label{d2.2}
\xymatrix{
&Y \ar[dl]^{2:1}_\pi \ar[dr]^{\psi}_{4:1}\\
X \ar[dr]^{4:1}_k && O \ar[dl]_{2:1}^\omega\\
& \PP^1 &
}
\end{equation}
with $X := Y/\iota$.

The following lemma is easy to check.

\begin{lem} \label{fixedfree}
The involution $\iota:Y \ra Y$ is fixed-point free.
\end{lem}

Identifying the $\PP^1$ of diagram \eqref{d2.2} with the image of $h$, we have 
according to \cite[Lemma 2.3]{d}:

\begin{lem} \label{l2.2}
The cover $O \ra \PP^1$ is branched exactly at $h(p_1)$ and $h(p_2)$. 
\end{lem}

Here a branch point 
is a point over which the fibre of $O \ra \PP^1$ consists of 1 point, so the fibre might be a singular point
of $O$.  Note that $O \ra \PP^1$ is of degree  2, so Lemma \ref{l2.2} implies that $O$ is connected. 

\begin{prop} \label{conn}
Suppose the branch locus of $f$ and the ramification locus of $h$ are disjoint.
Then the curve $Y$ is connected.
\end{prop}

\begin{proof}
It was proven in \cite{w} that $Y$ consists of two connected components when $f$ is \'etale. In that case  all the elements of the 
monodromy group of $\widetilde k$ are given by even permutations. The existence of ramification 
points for $f$ adds also odd permutations to the monodromy group, and taking into account that $O$ is connected, this
implies that the curve $Y$ is connected.
\end{proof}

\begin{prop} \label{p5.1}
Suppose the branch locus of $f$ and the ramification locus of $h$ are disjoint. 
\begin{enumerate}
\item If $h \circ f$ is general, then the curve $Y$ is smooth;
\item If $h \circ f$ is special, the the curve $Y$ is smooth apart from $2$ nodes.
\end{enumerate} 
\end{prop}

\begin{proof}
The proof is very similar to the proof of \cite[Proposition on p.107]{w} where Welters proves 
that $Y$ is smooth when $f$ is an \'etale double cover of a $d$-gonal curve (with some assumptions on the ramification of $h$).
In particular, $Y$ is smooth for $d=3$.
As we will see, the main difference lies in the fact that,  if a fibre of $h$ contains both branch points $p_1$ and $p_2$ of $f$,
the curve $Y$ acquires two nodes; otherwise, $Y$ is smooth.

Let $D_i$ be the fibre of $h$ passing through $p_i$, we denote by the same letter the divisor
on $C$ as well as the corresponding point of $C^{(3)}$. Of course $D_1 = D_2$ if $h \circ f$ is special. 
Let $\widetilde D_i^j$ denote the points of $Y$ 
defined by $D_i$ for $j = 1, \dots, \nu_i$ ($\nu_i = 2$, respectively $\nu_i=4$, if $h \circ f$ is special,
respectively general). 
The same proof as in \cite{w} works for all points of 
$Y \setminus (\cup_{i,j} \widetilde D_i^j)$. We shall show that $Y$ is smooth at the points 
$\widetilde D_i^j$ if $h \circ f$ is general and nodal if it is special.

So let $D$ be one of the points $D_i \in C^{(3)}$ and $\widetilde D \in \widetilde C^{(3)}$ a point above 
it. Suppose 
$$
D = p + q + r \quad \mbox{and} \quad \widetilde D = \widetilde p + \widetilde q + \widetilde r.
$$
where $p$ is one of the branch points of $f$, so $f^*(p) = 2\tilde{p}$.
The Zariski  tangent spaces yield the cartesian diagram
\begin{equation} \label{d5.1}
\xymatrix{
T_Y(\widetilde D) \ar@{^{(}->}[r] \ar[d]_{df} & T_{\widetilde C^{(3)}}(\widetilde D) \ar[d]^{df^{(3)}}\\
T_{g_3^1}(D) \ar@{^{(}->}[r]_{dj} & T_{c^{(3)}}(D)
}
\end{equation} 
where $j: g_3^1 \hra C^{(3)}$ denotes the inclusion map. The curve $Y$ will be smooth at 
$\widetilde D$ if and only if $\dim T_Y(\widetilde D) = 1$. 
It is easy to see 
(\cite[p. 104]{w}) that this is the case if and only if 
\begin{equation} \label{e5.2}
T_{C^{(3)}}(D) = dj T_{g_3^1}(D) + df^{(3)} T_{\widetilde C^{(3)}}(\widetilde D)
\end{equation}

According to deformation theory the lower-right triangle of diagram \eqref{d5.1} is given by
\begin{equation} \label{d5.2}
\xymatrix{
&& H^0(\cO_{\widetilde D}(\widetilde D) ) \ar[d]^\beta \\
H^0(\cO_C(D))/H^0(\cO_C) \ar@{^{(}->}[rr]^(.6)\alpha && H^0(\cO_D(D))
}
\end{equation}
where $\alpha$ is the canonical map. In order to define the map $\beta$ precisely, 
consider $\cO_C(D)$ (respectively $\cO_{\widetilde C}(\widetilde D)$) as a subsheaf of the rational function field $\cR_C$ of the curve $C$ (respectively $\cR_{\widetilde C}$ of the curve $\widetilde C$)
by putting for each point $p \in C$ and similarly for each $\widetilde p \in \widetilde C$, 
$$
\cO_C(D)_p := \mathfrak{m}_{C,p}^{-\nu_p(D)} \quad \mbox{and} \quad 
	\cO_{\widetilde C}(\widetilde D)_{\widetilde p} := \mathfrak{m}_{\widetilde C,\widetilde p}^{-\nu_{\widetilde p}(\widetilde D)}
$$
where $\mathfrak{m}_{C,p}$ respectively $\mathfrak{m}_{\widetilde C, \widetilde p}$ is the maximal ideal in 
$\cO_{C,p}$ respectively $\cO_{\widetilde C, \widetilde p}$. Translating diagram \eqref{d5.2}
to $\cO_C$ and $\cO_{\widetilde C}$ gives the diagram
$$
\xymatrix{
&& \mathfrak{m}_{\widetilde C,\widetilde p}^{-1}/\cO_{\widetilde C, \widetilde p} \oplus 
\mathfrak{m}_{\widetilde C,\widetilde q}^{-1}/\cO_{\widetilde C, \widetilde q} \oplus
\mathfrak{m}_{\widetilde C,\widetilde r}^{-1}/\cO_{\widetilde C, \widetilde r} \ar[d]^\beta\\
H^0(\cO_C(D))/H^0(\cO_C) \ar@{^{(}->}[rr]^(.4)\alpha &&  
\mathfrak{m}_{ C, p}^{-1}/\cO_{ C, p} \oplus \mathfrak{m}_{ C, q}^{-1}/\cO_{ C, q} \oplus
\mathfrak{m}_{ C, r}^{-1}/\cO_{ C, r} 
}
$$

where  $\beta$ is given by the transposition of the natural map 
$$
\beta^{*}: \Omega^1_{C,p} / \mathfrak{m}_{C,p}^{1} \oplus  \Omega^1_{C,q} / \mathfrak{m}_{C,q}^{1} \oplus \Omega^1_{C,r} / \mathfrak{m}_{C,r}^{1}
\ra \Omega^1_{\tC,\tilde p} / \mathfrak{m}_{\tC,\tilde p}^{1} \oplus  \Omega^1_{\tC,\tilde q} / \mathfrak{m}_{\tC,\tilde q}^{1} \oplus \Omega^1_{\tC,
\tilde r} / \mathfrak{m}_{C,r}^{1}
$$
induced by $f$. Let $t_p, t_q, t_r$ be local parameters around $p,q$ and $r$ respectively.
Then $ t_{\tilde p} := t^{\nu(p)} = t^2$ is a local parameter around  $\tilde p$, since $f$ is has
ramification index 2 at $\tilde p$.  We claim that $\beta^{*}$ restricted to the summand  $\Omega^1_{C,p} / \mathfrak{m}_{C,p}^{1} $ is
not injective, hence $\beta$ restricted to $\Omega^1_{\tC,\tilde p} / \mathfrak{m}_{\tC,\tilde p}^{1} $ cannot be surjective.
Indeed, if $dt_p$ is a generator of the one-dimensional space $\Omega^1_{C,p} / \mathfrak{m}_{C,p}^{1} $, then 
$$
\beta^{*} (dt_p) = dt_{\tilde p} =  d(t_p^2) = 2t_p dt_p 
$$
vanishes modulo $\mathfrak{m}_{C, p}^{1} $, therefore the restriction is zero. In particular, the restriction of  $\beta$
to $\Omega^1_{\tC,\tilde p} / \mathfrak{m}_{\tC,\tilde p}^{1} $ is zero, since the target is one-dimensional. 

\bigskip

{\it (1) Suppose $h\circ f$ is general}.  In this case $f$ is \'etale at $\tilde q$ and $\tilde r$.
Let $t_{\tilde q}$ and   $t_{\tilde r}$ local parameters around $\tilde q$ and $\tilde r$. 
We have that  
$$
\beta ( t_{\tilde p}) =0,  \qquad \beta ( t_{\tilde q}) = t_q, \qquad \beta ( t_{\tilde r}) = t_r.
$$
Then $\beta $ is surjective onto the summands $\mathfrak{m}_{ C, q}^{-1}/\cO_{ C, q} $ and $\mathfrak{m}_{ C, r}^{-1}/\cO_{ C, r}$,
but not onto the first one, where $\beta$ vanishes.

The curve $Y$ is smooth at $\tilde D$ if and only if the composition of $\alpha$ with the cokernel of $\beta$ is surjective,
i.e. if and only if the map
\begin{equation} \label{surj-map1}
H^0(\cO_C(D))/H^0(\cO_C) \ra \mathfrak{m}_{ C, p}^{-1}/\cO_{ C, p} 
\end{equation}
is surjective. As a vector space the target is generated by $t_p^{-1} $. Let  $\psi \in H^0(\cO_C(D))$ be a section with corresponding
divisor $\bar D= \divi{\psi} +D $. Its image $\bar \psi$ in $\mathfrak{m}_{C,p}^{-1}/\cO_{ C, p} $ can be written 
as  $\bar \psi = c_k t^{-k}$ with  $c_k \neq 0. $ So the map \eqref{surj-map1} is surjective if such function exists with $k = 1$.
As $\nu_p(\bar D) = \nu_p (\psi) + \nu_p(D) = - k +  1 \geq 0$, it is enough to take a divisor $\bar D$ with $\nu_p(\bar D)=0$, i.e.
which does not contain $p$.  Therefore $Y$ is smooth at $\tilde D$.  

\bigskip

{\it (1) Suppose $h\circ f$ is special}. In this case we can assume $p=p_1$ and $q=p_2$  are the branch points of $f$. Then
$$
\beta ( t_{\tilde p}) = 0, \qquad \beta ( t_{\tilde q}) =0, \qquad \beta ( t_{\tilde r}) = t_r
$$
and  $Y$ is smooth at $\tilde D$ if and only if the map
\begin{equation}\label{surj-map}
H^0(\cO_C(D))/H^0(\cO_C) \ra \mathfrak{m}_{ C, p}^{-1}/\cO_{ C, p} \oplus  \mathfrak{m}_{ C, q}^{-1}/\cO_{ C, q}
\end{equation}
is surjective. Now the target space has as basis the vectors $(t_p^{-1}, 0 ), (0, t_q^{-1})$, so for the surjectivity one requires
the existence of a section $\psi \in H^0(\cO_C(D))$ such that its corresponding divisor $\bar D$ satisfies
$$
\nu_p(\bar D) = \nu_p (\psi) + \nu_p(D) = -1 + 1 =0, \qquad  \nu_q(\bar D) = \nu_q (\psi) + \nu_q(D) = 0 + 1 =1,
$$
that is, a divisor $\bar D$ in the fibre of $h$ containing $q$ but not $p$, which gives a contradiction since the linear series $g^1_3$
defining $h: C \ra \PP^1$ is base-point-free\footnote{Assuming that $C$ is non-hyperelliptic}.

In consequence, $Y$ has 2 nodal singularities at the divisors $\tilde D_1^1 $ and $\tilde D_1^2 $ over $D=p_1 +p_2 +r$.
\end{proof}

\begin{prop} \label{genus}
Assume $h\circ f$ is a general, then $g_Y = 2g+1$.
\end{prop} 

\begin{proof}
There are 2 types of ramification of $\widetilde k: Y \ra \PP^1$ for a branch point $a$ of 
$\widetilde k$. Either $a$ is a 
branch point of $h$ of ramification index $i = 1$ or 2. Then $a$ is a branch point for $h$ with 
ramification of type $(2,2,1,1,1,1)$ if $i =1$ and $(3,3,1,1)$ if $i =2$.  Or $a = h(p_\nu)$ for 
$\nu = 1$ or 2. Then  $a$ is a branch point for $h$ with ramification of type $(2,2,2,2)$.  
Since $|R_h| = 2g+4$ and $p_1$ and $p_2$ are disjoint from $R_h$, this gives for the ramification divisor 
$R_{\widetilde k}$ of $\widetilde k$: 
$$
|R_{\widetilde k}| = 2(2g+4) + 2 \cdot 4 = 4g + 16
$$
and the result follows from the Hurwitz formula.
\end{proof}

We will see in Section 4, that  in the case of special coverings, $Y$ has arithmetic genus $2g+1$, more precisely,
it is the union of two curves of genus $g$ intersecting in two points.

\section{Donagi's extension of the trigonal construction}

The usual trigonal construction for \'etale double covers of smooth trigonal curves is the following theorem, due to Recillas 
\cite{r}. We recall it for \'etale double covers $f$ of trigonal curves $C'$ of genus $g+1$
instead of genus $g$, since we need it in this case. 

Let $\cR_{g+1}^{tr}$ denote the moduli space of \'etale double covers of smooth
trigonal curves of genus $g+1$,  so it consists of triples $(C, \eta, g^1_3)$, where $\eta \in \Pic^0(C)[2] \setminus \{ 0\}$
(which defines the double covering over $C$)  and $g^1_3$ is a linear series on $C$. Let $\mathcal{G}^1_{g,4}$ denote the moduli
space of pairs $(C, g^1_4)$, with $C$ a smooth curve of genus $g$ together with a  $g^1_4$, and  $\cM_{g,0}^{tet} \subset \mathcal{G}^1_{g,4} $ 
denote the open subspace of tetragonal curves $X$ of genus $g$ with the property that above each point of 
$\PP^1$ there is at least one \'etale point of $X$. In the sequel trigonal and tetragonal curves are understood as a pair of a curve with its linear series.   
In (\cite{r}) Recillas showed 

\begin{thm}
The trigonal construction gives an isomorphism
$$
T^0: \cR_{g+1}^{tr} \ra \cM_{g,0}^{tet}.
$$
Moreover, for each $[h \circ f] \in \cR_{g+1}^{tr}$, $T^0$ induces an isomorphism of principally polarized 
abelian varieties
$$
\Pr(f) \simeq J(T^0(h \circ f)).
$$
\end{thm}

\begin{rem}
Recall that  curves of genus $g\geq 5$ (respectively  $g\geq 7$) admit at most  a $g^1_3$  (respectively a  $g^1_4$) and they form 
a closed subset in the moduli of smooth curves $\cM_g$.  So for high genus one can identify the moduli  space  $ \cM_{g,0}^{tet}$
with its corresponding closed subspace in $\cM_g $ and something similar holds for $\cR_{g+1}^{tr}$.
On the other hand, any curve of genus $g\leq  4$ (respectively  $g\leq 6$)  possesses a $g^1_3$ (respectively  a $g^1_4$) and in this case
the isomorphism $T_0$ sends triples $(C, \eta,  g^1_3)$ to pairs $(X, g^1_4)$.
\end{rem}

In \cite{d} Donagi extended the map $T^0$ to the partial compactification 
consisting of admissible double covers of trigonal curves of genus $g+1$, whose Prym variety is an element in $\cA_g$, the moduli
space of principally polarized abelian varieties of dimension $g$. 
Recall that an {\it admissible double cover} of a connected stable curve $C'$  is a double cover $\pi: \widetilde C' \ra C'$  with $\widetilde C'$ stable such that
\begin{itemize}
\item the nodes of $\widetilde C'$ map under $\pi$ exactly to the nodes of $C'$;
\item away from the nodes $\pi$ is an \'etale double cover.
\end{itemize}
Let $\overline \cR_{g+1}^{tr}$ denote the moduli space of admissible double covers of curves of 
arithmetic genus  $g+1$, whose  Prym variety is in $\cA_g$\footnote{In Donagi's work  these coverings are called ``allowable''.
Often the coverings with this property are also referred to as  ``admissible''. We use this notion.}
and let $\cM_{g}^{tet}$ denote the moduli space of smooth tetragonal curves  of genus $g$.
Donagi showed in \cite[Theorem 2.9]{d},

\begin{thm}
The trigonal construction gives an isomorphism 
$$
T: \overline \cR_{g+1}^{tr} \ra \cM_{g}^{tet}.
$$
\end{thm}
Moreover, he showed that if $\widetilde C' \stackrel{f'}{\ra} C' \stackrel{h'}{\ra} \PP^1$ corresponds to
the tetragonal curve $k': X' \ra \PP^1$ and $p$ is a point of $\PP^1$, then the following types of fibers over $p$
of $f', k'$ and $h'$ correspond to each other,
\begin{enumerate}
\item $f',k'$ and $h'$ are \'etale;
\item $k'$ and $h'$ have simple ramification points, $f'$ is \'etale;
\item $k'$ and $h'$ each have a ramification point of index 2, $f'$ is \'etale;
\item $k'$ has two simple ramification points, the fibre of $h'$ consists of a simple node and a smooth 
point, $f'$ is admissible;
\item $k'$ has a ramification point of index 3, the fibre of $h'$ consists of a node and at exactly one 
branch of it $h'$ is ramified, $f'$ is admissible.
\end{enumerate}

In particular, if $h'$ admits nodes of types (4) or (5), the covering $f'$ is admissible. By \cite{b},
its Prym variety, defined again as the component of the norm map containing 0, is a 
principally polarized abelian variety of dimension $g$, where $g+1$ is the arithmetic genus of 
$C'$. 

The {\it trigonal construction} is formally the same as in the smooth case.
Let $\widetilde C' \stackrel{f'}{\ra} C' \stackrel{h'}{\ra} \PP^1$ be an element of 
$ \overline \cR_{g+1}^{tr}(2)$ and define $Y'$ in the same way as we defined $Y$ in diagram \eqref{d2.1}. The involution of $\widetilde C'$ induces an involution 
on $Y'$ of which $X'$ is the quotient. (Donagi defines $X'$ as follows: The restriction 
$f'^0: \widetilde C'^0 \ra C'^0$ of $f'$ to the smooth parts can be considered as a local system on $C'^0$. 
Its push-forward $h'_*  \widetilde C'^0$ to $\PP^1$ admits an involution $\iota$. Then $X'$ is defined as the completion of the quotient $h'_*  \widetilde C'^0/\langle \iota \rangle$).

Conversely, the inverse $T^{-1}$ of $T$ is given as follows. Given $k': X' \ra \PP^1$ in 
$\cM_g^{tet}$, we denote by  $k'_0: X'_0 \ra \PP^1_0$ its restriction to  the open part consisting of fibres 
of  types (1), (2) and (3). Let $\widetilde C'$ denote the closure of $ S^2_{\PP^1}X'_0$,  the relative second symmetric product 
of $X'_0$ over $\PP^1$. Any fibre over a point 
$p \in \PP^1_0$ consists of all unordered pairs in $k'^{-1}(p)$ and $\widetilde C'$ admits an obvious 
involution of which $C'$ is the quotient. Moreover, the cover $\widetilde C' \ra \PP^1$ factorizes via  
maps $f': \widetilde C' \ra C'$ and $h':C' \ra \PP^1$. The tower 
$\widetilde C' \stackrel{f'}{\ra} C' \stackrel{h'}{\ra} \PP^1$ is then an element of  
$\overline \cR_{g+1}^{tr}$.

The same proof as in the smooth case gives an isomorphism
\begin{equation} \label{e3.1}
\Pr(f') \simeq J(T(h' \circ f')).
\end{equation}

In the next section we need the following corollary. Let 
$$
\cS_{g+1}^{tr} \subset  \overline \cR_{g+1}^{tr}
$$ 
the locally closed subset
of $ \overline \cR_{g+1}^{tr}$ consisting of all elements of type (4), i.e. all admissible double 
covers $f': \widetilde C' \ra C'$ of 
trigonal  curves $C'$ of arithmetic genus $g+1$ such that $C'$ admits exactly one node such that the 
fibre of $h'$ containing the node contains also a smooth point.  Let 
$\cM_{g,1}^{tet}$ be the corresponding 
subset of tetragonal covers $g': X' \ra \PP^1$ with exactly one fibre consisting of two simple ramification points and otherwise only fibres of types 
(1), (2) and (3). So we get

\begin{cor} \label{c3.3}
Restricting the trigonal construction to $\cS_{g+1}^{tr}$ gives an isomorphism 
$$
T: \cS_{g+1}^{tr} \ra \cM_{g,1}^{tet}
$$
inducing for every element of $\cS_{g+1}^{tr}$ the isomorphism \eqref{e3.1}.
\end{cor}

\section{Proof of the main theorem for special ramified covers}

Let $C$ be a smooth trigonal curve of genus $g$ with trigonal cover $h:C \ra \PP^1$ and let 
$f: \widetilde C \ra C$ be a special ramified double cover, i.e. the two branch points of $f$ are 
contained in a fibre of $h$. Suppose $p_1, p_2 \in C$ are the branch points  of $f$ and 
$q_1, q_2 \in \widetilde C$ the corresponding ramification points. Let 
$\cR b^{tr}_{g,sp}$ denote the closed subset of the variety $\cR b^{tr}_g$ consisting of special
covers.

Define new curves $\widetilde C'$ and $C'$ by identifying the points:
$$
\widetilde C' := \widetilde C/(q_1 \sim q_2) \quad \mbox{and} \quad C' := C/(p_1 \sim p_2).
$$ 
The covering $f$ induces a covering $f': \widetilde C' \ra C'$ which clearly is an admissible double
cover. Similarly $h$ induces a trigonal cover $h': C' \ra \PP^1$, such that the tower 
$\widetilde C' \stackrel{f'}{\ra} C' \stackrel{h'}{\ra} \PP^1$ is an element of $\cS_{g+1}^{tr}$.
So we get a map
$$
n: \cR b^{tr}_{g,sp} \ra \cS_{g+1}^{tr}.
$$
\begin{lem} \label{l4.1}
The map $n: \cR b^{tr}_{g,sp} \ra \cS_{g+1}^{tr}$ is an isomorphism.
\end{lem}
\begin{proof}
Normalization gives the inverse map.
\end{proof}

\begin{lem} \label{l4.2}
Let $\widetilde C \stackrel{f}{\ra} C \stackrel{h}{\ra} \PP^1$ be an element of  $\cR b^{tr}_{g,sp}$
and $\widetilde C' \stackrel{f'}{\ra} C' \stackrel{h'}{\ra} \PP^1$ the corresponding element of 
$\cS_{g+1}^{tr}$. Normalization induces the following isomorphism of principally polarized 
abelian varieties
$$
N: P(f') \ra P(f).
$$
\end{lem}

\begin{proof}
According to the proof of \cite[Proposition 3.5]{b} the norm map induces the following commutative diagram with exact rows and columns 
$$
\xymatrix{
& 0 \ar[d] & 0 \ar[d] & 0 \ar[d] &\\
0 \ar[r] & \ZZ/2 \ar[r] \ar[d] & P(f') \times \ZZ/2 \ar[r] \ar[d] & P(f) \ar[r] \ar[d] & 0\\
0 \ar[r] & \CC^* \ar[r] \ar[d]_\Nm & J\widetilde C' \ar[r] \ar[d]_\Nm & J\widetilde C \ar[r] \ar[d]_\Nm & 0\\
0 \ar[r] & \CC^* \ar[r] \ar[d] & JC' \ar[r] \ar[d] & JC \ar[r] \ar[d] & 0\\
& 0 & 0 & 0 &
}
$$
This gives the assertion.
\end{proof}

Combining Corollary \ref{c3.3} with Lemma \ref{l4.1} and equation \eqref{e3.1} with Lemma 
\ref{l4.2} gives the following theorem

\begin{thm} \label{thm4.3}
The map 
$$
T \circ n:  \cR b^{tr}_{g,sp} \ra  \cM_{g,1}^{tet}
$$
is an isomorphism.

Moreover, if $\widetilde C \stackrel{f}{\ra} C \stackrel{h}{\ra} \PP^1$ is an element of  $\cR b^{tr}_{g,sp}$
and $X'$ is the corresponding (smooth) tetragonal cover, then we have an isomorphism of principally polarized abelian 
varieties
$$
P(f) \ra JX'.
$$
\end{thm}

\begin{rem} \label{rem4.4}
\begin{enumerate}
\item[(i)] 
In \cite{d}, Donagi outlines also explicitely the inverse map $T^{-1}$. Since the inverse of the map 
$n$ is clear, this gives an explicit version of the inverse map of Theorem \ref{thm4.3}. 
\item[(ii)]
One can also express $X'$ as the Prym variety of an \'etale double cover of $X'$. In fact, the double 
cover $Y' \ra X'$ which occurred in Corollary  \ref{c3.3} is trivial. Hence $JX' = P(Y'/X')$. 
\end{enumerate}
\end{rem}

\vspace{1cm}

Finally, we shall describe the covering $Y \ra X$ with $X$ and $Y$ defined as in Section 2. As a consequence, it shows that $P(Y/X)$
is an abelian variety of the right  dimension.

Let $D := p_1 + p_2 +r$ be the unique divisor in $g^1_3$ containing the 2 branch points of $f$.
Denote by
$$
\widetilde k:= k \circ \pi:  Y \ra \PP^1
$$
the $8:1$ covering of diagram \eqref{d2.2}.
If $r_1$ and $r_2 = \iota(r_1)$ are the 2 points of $f^{-1}(r)$, then 
$$
D_1 := q_1 + q_2 +r_1  \in Y  \quad \mbox{and} \quad D_2 := q_1 + q_2 +r_2 \in Y
$$
are exactly the 2 points of $\widetilde k^{-1}(h(p_1))$.

\begin{lem}
The curve $Y$ consists of $2$  smooth irreducible components
$$
Y = Y_1 \cup Y_2
$$
with 
$$
Y_2 = \iota(Y_1), \quad \quad Y_1 \cap Y_2 = \{D_1,D_2\} \quad \mbox{and} \quad \iota(D_1) = D_2.
$$
\end{lem}

\begin{proof}
According to Lemma \ref{l2.2} the cover $\omega: O \ra \PP^1$ is branched exactly at $h(p_1) =
h(p_2)$. Since $\PP^1$ is simply connected, this implies that $O$ consists of 2 components
$O = O_1 \cup O_2$ with $O_i \simeq \PP^1$ intersecting in the point lying over $h(p_i)$.

It follows that $Y = \psi^{-1}(O)$ consists of at least 2 irreducible components. But it cannot consist 
of more than 2 components, since the open dense part $\widetilde k^{-1}(\PP^1 \setminus h(p_1))$
is isomorphic to an open dense part of the cover $Y' \ra X'$ of Remark \ref{rem4.4}, which consists of exactly 2 components $Y_1$  and $Y_2$.

Clearly $Y_1$ and $Y_2$ are smooth with $Y_1 \cap Y_2 = \{D_1,D_2\}$ and $\iota(Y_1) = Y_2$. 
\end{proof}

\begin{cor}
The double cover
$$
Y \ra X:= Y/\iota
$$
is a Wirtinger cover in the sense of \cite[Example 1.9,(I)]{d}. In particular its Prym variety $P(Y/X)$
is an abelian variety.
\end{cor}

\begin{proof}
This follows from the fact that $\iota:Y \ra Y$ is fixed-point free by Lemma \ref{fixedfree}
and $\iota(D_1)=D_2$.
\end{proof}

\begin{lem} We have
$$
p_a(Y) = 2g+1 , \qquad p_a(X) = g+1.
$$
and hence $\dim P(Y/X) = p_a(Y) - p_a(X) = g$.
\end{lem}

\begin{proof}
For $i = 1$ and 2 the map $\psi_i:= \psi|Y_i: Y_i \ra O_i = \PP^1$ is a 4-fold cover, where we
identify $O_i$ with $ \PP^1$ via the map $\omega$. The  map $\psi_i$ can be considered as the map
$\widetilde k_i := \widetilde k|Y_i: Y_i \ra \PP^1$.

We then have 
$$
p_a(Y) = 2g(Y_i) +1
$$
and it suffices to show that $g(Y_i) = g$. Then also $p_a(X) = g+1$, since $X$ admits exactly 1 node and $Y_i$ is the normalization of $X$.

There are 2 types of ramification of $\widetilde k_i: Y_i \ra \PP^1$ for a branch point $a$ of 
$\widetilde k_i$. Either $a$ is a 
branch point of $h$ of ramification index $i = 1$ or 2. Then $a$ is a branch point for $\widetilde k_i$ 
with ramification of type $(2,1,1)$ if $i =1$ and $(3,1)$ if $i =2$.  

Or $a = h(p_1) = h(p_2).$ Then  $a$ is a branch point for $\widetilde k_i$ with ramification of type $(2,2)$, i.e. $\widetilde k_i^{-1}(a) = 
\{2 \widetilde D_1, 2 \widetilde D_2\}$.
 
Since $|R_h| = 2g+4$ and $p_1$ and $p_2$ are disjoint from $R_h$, this gives for the ramification divisor $R_{\widetilde k_i}$ 
of $\widetilde k_i$: 
$$
|R_{\widetilde k_i}| = 2g+4 + 2  = 2g + 6.
$$
and the Hurwitz formula gives $g(Y_i) = g.$
\end{proof}

\section{The main theorem for general covers}

Let $C$ be a smooth trigonal curve of genus $g$ with trigonal cover $h: C \ra  \PP^1$ and let 
$f: \widetilde C \ra C$ be a general ramified double cover branched over 2 points $p_1,p_2 \in C$,
i.e. $p_1$ and $p_2$ do not lie in a fibre of $h$ and are disjoint from the ramification locus of $h$.
Let $Y$ be the curve defined by diagram \eqref{d2.1}. According to Propositions \ref{conn}, \ref{p5.1} and  \ref{genus}, 
$Y$ is smooth and irreducible of genus $2g+1$.
According to Proposition \ref{fixedfree} it admits a fixed-point free involution $\iota$. As above let
$\pi: Y \ra X$ denote the corresponding \'etale double cover with diagram \eqref{d2.2}. 
Let $(P(\pi),\Xi_\pi)$ denote the corresponding principally polarized Prym variety of dimension $g$.
Let $\cR b^{tr}_{g,gen}$ and $\cR_{g+1}^{tet}$ the moduli spaces as defined in the introduction.

\begin{thm} \label{thm5.1}
 There is a canonical map 
$$
\cR b^{tr}_{g,gen} \ra  \cR_{g+1}^{tet}.
$$
If $\widetilde C \stackrel{f}{\ra} C \stackrel{h}{\ra} \PP^1$ is an element of  $\cR b^{tr}_{g,gen}$
and $\pi: Y \ra X$ the corresponding \'etale double cover, then the principally polarized abelian varieties
$(P(f), \Xi_f)$ and $(P(\pi), \Xi_\pi)$ are canonically isomorphic.

\end{thm}

\begin{proof} {\bf Step 1}: {\it The set up.}

Choose a point $y_0 \in Y$ and let $\widetilde \alpha = \widetilde  \alpha_{y_0}: \widetilde C^{(3)}\ra J\tC$ and 
$\alpha = \alpha_{f(y_0)}: C^{(3)} \ra JC$ be the corresponding Abel-Jacobi maps. Then the following 
diagram 
is commutative
\begin{equation} \label{d6.1}
\xymatrix{
Y \ar@{^{(}->}[r] & \widetilde C^{(3)} \ar[r]^{\widetilde \alpha} \ar[d]_{f^{(3)}} & J\widetilde C \ar[d]^{\Nm_f}\\
& C^{(3)} \ar[r]^\alpha & JC
}
\end{equation}
This shows that 
$$
\widetilde \alpha(Y) \subset \ker \Nm_f = P(f),
$$
since $\alpha$ maps $f^{(3)}(Y) = g_3^1 = \PP^1$ to 0 in $JC$. 
Denote by $\varphi$ the restriction $\widetilde \alpha|Y$ as a map into $P(f)$,
$$
\varphi := \widetilde \alpha|Y: Y \ra P(f).
$$
For any 
$y = \sum_{i=1}^3 \widetilde c_i \in Y$ we have
\begin{eqnarray*}
(\varphi + \varphi \iota)(y)& = &  \cO_Y\left( y
+ \iota (y) -2y_0  \right)\\
& =& \cO_Y\left( \sum_{i=1}^3 \widetilde c_i 
+ \sum_{i=1}^3 \iota(\widetilde c_i) -2y_0  \right) = \cO_Y(f^*g_3^1 -2y_0).\\
\end{eqnarray*}
But this is a constant in $P(f)$. So replacing $\alpha$ by a suitable translate we may assume that we have
$$
\varphi \iota =  - \varphi.
$$
According to the universal property for Prym varieties (valid also for ramified double covers, see 
\cite[12.5.1]{bl}) the map $\varphi$ factorizes via the Abel-Prym map $\psi: Y \ra P(\pi)$, i.e. the following diagram is commutative
$$
\xymatrix{
Y \ar[r]^\varphi \ar[d]_\psi & P(f) \ar[d]^{t_{-\varphi(y_0)}}\\
P(\pi) \ar[r]^{\overline \varphi} & P(f)
}
$$
where $\overline \varphi$ denotes the extension of $\varphi$ and the right vertical map the translation by $-\varphi(y_0)$.\\

{\bf Step 2}: {\it We claim that for the proof of Theorem \ref{thm5.1} it suffices to show that}
\begin{equation} \label{e6.1}
\varphi_*[Y] = \frac{2}{(g-1)!} \wedge^{g-1}[\Xi_f] \quad \mbox{in} \quad H^{2g-2}(P(f),\ZZ).
\end{equation}

\begin{proof}
Note first that $\varphi(Y)$ generates $P(f)$ as an abelian variety. This implies that 
$\overline \varphi$ is an isogeny, $P(\pi)$ and $P(f)$ being of the same dimension. According to 
\cite[Welters' criterion 12.2.2]{bl} we have 
$$
\psi_*[Y] = \frac{2}{(g-1)!} \wedge^{g-1}[\Xi_\pi] \quad  \mbox{in} \quad  H^{2g-2}(P(\pi),\ZZ).
$$
Since we assume \eqref{e6.1}, the commutativity of the diagram gives 
$$
\wedge^{g-1} [\Xi_\pi] = \wedge^{g-1} [\Xi_f].
$$ 
But then \cite[Lemma 12.2.3]{bl} implies that $\overline \varphi$ is an isomorphism of polarized 
abelian varieties.
\end{proof}

{\bf Step 3}: {\it Let $i_{P(f)}: P(f) \hra J \widetilde C$ denote the canonical embedding.
Then}
$$
\widehat i_{P(f)*} \widetilde \alpha_* [Y] = \frac{8}{(g-1)!} \wedge^{g-1} [\Xi_f].
$$

The proof applies the formula of Macdonald for the class of the curve
$g_3^1$ in the variety $C^{(3)}$ in $H^4(C^{(3)},\ZZ)$ (see \cite[Lemma 8.3.2]{acgh}): 
$$
[g_3^1] = \sum_{k=0}^2 {2-g \choose k} \eta^k \cdot \frac{\wedge^{2-k}[\alpha^*\Theta]}{(2-k)!} \quad \mbox{in} 
\quad H^4(C^{(3)},\ZZ)
$$
where $\eta^k \in H^{2k}(C^{(3)},\ZZ)$ denotes the fundamental class of $C^{(3-k)}$ 
in $C^{(3)}$ under the 
embedding $\sum_{k=1}^{3-k} p_i \mapsto \sum_{k=1}^{3-k} p_i+ kp$ with some fixed point $p \in C$.
Similarly we define the classes $\widetilde \eta^k \in H^{2k}(\widetilde C^{(3)},\ZZ)$.
It is well known that $\eta^k = \wedge^k \eta^1$ and $\widetilde \eta^k = \wedge^k \widetilde \eta^1$ and both are related by
$$
{f^{(3)}}^* \eta^k = 2^k \widetilde \eta^k.
$$
Applying this to Macdolnald's formula, we get with the commutativity of diagram \eqref{d6.1}
\begin{equation} \label{e6.3}
[Y] = {f^{(3)}}^*[g_3^1] = \sum_{k=0}^2 {2-g \choose k}2^k \widetilde \eta^k \cdot \frac{\widetilde \alpha^* 
\Nm_f^*\wedge^{2-k}[\Theta]}{(2-k)!} \quad \mbox{in} \quad H^4(\widetilde C^{(3)},\ZZ).
\end{equation}
Now Poincar\'e's formula \cite[11.2.1]{bl} in this case says
$$
\widetilde \alpha_* \widetilde \eta^k = \frac{\wedge^{2g+k-3} [\widetilde \Theta]}{(2g+k-3)!}.
$$
Applying $\widetilde \alpha_*$ to \eqref{e6.3}, Poincar\'e's formula and the projection formula give
\begin{equation} \label{e6.4}
\widetilde \alpha_*[Y] = \sum_{k=0}^2 {2-g \choose k} 2^k \frac{\wedge^{2g+k-3} [\widetilde \Theta]}{(2g+k-3)!} \cdot  
\frac{ \Nm_f^*\wedge^{2-k}[\Theta]}{(2-k)!} 
\quad \mbox{in} \quad H^{4g-2}(J\widetilde C,\ZZ).
\end{equation} 
According to \cite[Proposition 12.3.4]{bl} we have
$$
2 [\widetilde \Theta] = \Nm_f^*[\Theta] + \widehat \iota_{P(f)}^*[\Xi_f].
$$
 Applying this and the binomial formula to \eqref{e6.4}, we get
\begin{equation} \label{e6.5}
\widetilde \alpha_* [Y] = 2^{3-2g} \sum_{k=0}^2 {2-g \choose k} \sum_{j=0}^{2g+k-3}
{2g+k-3 \choose j} \frac{\Nm_f^* \wedge^{j+2-k}[\Theta]}{(2-k)!} \cdot \frac{\widehat \iota_{P(f)}^* \wedge^{2g+k-3-j}[\Xi_f]}{(2g+k-3)!}
\end{equation}
Now we claim that
$$
\widehat \iota_{P(f)_*} \Nm_f^* \bigwedge^n \frac{[\Theta]}{n!} = \left\{ \begin{array}{ccl} 
                                                                                              2^{2g} & if & n= g,\\
                                                                                               0   &if& 0\leq n \leq g-1.
                                                                                              \end{array}  \right.
$$
For this consider the composed map
$$
\widehat  \iota_{P(f)_*} \circ \Nm_f^*: H^{2n}(JC,\ZZ) \ra H^{2n}(J \widetilde C,\ZZ) \ra H^{2n-2g}(P(f),\ZZ).
$$
By degree reasons, $\widehat  \iota_{P(f)_*} \Nm_f^* \equiv 0$ for $n \neq g$. For $n=g$ note that
$\widehat \iota_{P(f)}$ is multiplication by 2 on $P(f)$. So we get, denoting by $[0]$ the class of the point 0 in $JC$,
$$
\widehat  \iota_{P(f)*} \Nm_f^* \frac{\wedge^{g}[\Theta]}{g!} = 
\widehat  \iota_{P(f)*} \Nm_f^* [0] = \widehat  \iota_{P(f)*} [P(f)] = \deg (2_{P(f)}) = 2^{2g}.
$$
Inserting this into equation \eqref{e6.5} we get
\begin{eqnarray*}
\widetilde \iota_{P(f)*} \widetilde \alpha_* [Y] & =& 2^{3-2g} \sum_{k=0}^2 {2-g \choose k} {2g+k-3 \choose g+k -2}
\frac{2^{2g} g!}{(2-k)!} \cdot \frac{\wedge^{g-1} [\Xi_f]}{(2g+k-3)!}\\
&=& \frac{2^{3}}{(g-1)!} \wedge^{g-1} [\Xi_f] \sum_{k=0}^2   {2-g \choose k} 
{g \choose 2-k}\\
&=& \frac{2^{3}}{(g-1)!} \wedge^{g-1} [\Xi_f],
\end{eqnarray*}
since $\sum_{k=0}^2   {2-g \choose k} {g \choose 2-k}= \frac{1}{2}g(g-1) + (2-g)g + \frac{1}{2} (2-g)(1-g) = 1$, which is also the coefficient of 
$x^2$ in the product $(1+x)^{2-g}(1+x)^g$.

{\bf Step 4}: Now $\alpha_*(Y) \subset P(f)$ and $\widehat \iota_{P(f)}$ restricted to $P(f)$ is multiplication by 2, so
$$
\widetilde \iota_{P(f)*} \widetilde \alpha_* [Y] = 2_{P(f)*}\widetilde \alpha_*[Y] = 4 \widetilde \alpha_*[Y] = \frac{8}{(g-1)!} \wedge^{g-1} [\Xi_f] 
\quad \mbox{in} \quad H^{2g-2}(P(f),\ZZ).
$$  
Since $H^{2g-2}(P(f),\ZZ)$ is torsion free, we obtain equation \eqref{e6.1}. This completes the proof of Theorem \ref{thm5.1}.
\end{proof}
\bigskip

\begin{prop}  \label{p5.2} 
Let  $\cR_{g+1,2}^{tet}$ be the locus of double coverings $[Y \ra X] \in \cR^{tet}_{g+1}$ such that the tetragonal covers $k: X \ra \PP^1$ have 
exactly two fibres consisting of two simple ramification points and otherwise  only fibres containing at least one smooth point of $X$. 
 Then the image of the canonical map of Theorem \ref{thm5.1} is contained in $\cR_{g+1,2}^{tet}$.  
\end{prop}

\begin{proof}
According to the proof of Proposition \ref{genus} the fibres over $h(p_\nu)$, the images of the 2 branch points of $f$, are the 
only ones of type $(2,2)$ in the corresponding tetragonal cover
$k:X \ra \PP^1$.
\end{proof}

\begin{rem} \label{r5.3}
According to Theorem \ref{thm5.1} and Proposition \ref{p5.2} the trigonal construction in the general ramified case is given by the map
$$
\cR b^{tr}_{g,gen} \ra \cR_{g+1,2}^{tet}.
$$
Both moduli spaces are of dimension $2g+3$. 
It would be interesting to compute the exact image as well as the degree of this map. We hope to come back to this question. 
\end{rem}

\begin{rem}
During the preparation of this paper we came across the article \cite{da}, where it is claimed that Prym variety of a double 
covering of a 
trigonal curve ramified at  two points is isomorphic as ppav to the Jacobian of certain tetragonal  curve $X$. The author constructs
a tower of curves,  which links the curves in the double covering and $X$. 
It seems to us that in \cite{da} all curves are supposed to be smooth, even in the case
of a special cover. So we do not know whether this construction agrees 
with our construction, meaning that it gives the same map of the corresponding moduli spaces.  Certainly they 
give isomorphic Prym varieties, but taking into account Donagi's tetragonal construction it is not clear 
that the coverings are the same.  
\end{rem}

\end{document}